\documentclass[runningheads]{llncs}
\usepackage[T1]{fontenc}
\usepackage{graphicx}
\usepackage{amsmath,  amssymb, amsfonts}
\usepackage{color}

\newcommand{\imp}[1]{\textcolor{black}{#1}}

\begin{document}
\title{On the Sprague-Grundy values of games with a pass}
%
%
\author{Hikaru Manabe\inst{1}\orcidID{0009-0000-1835-493X} \and
Ryohei Miyadera\inst{2}\orcidID{0009-0003-7708-5970} \and
Koki Suetsugu\inst{3,4,5}\orcidID{0000-0003-2529-7501}}
\authorrunning{H. Manabe et al.}
%
\institute{University of Tsukuba College of Informatics, 1-1-1 Tennodai, Tsukuba, Japan\\ 
\email{urakihebanam@gmail.com} \\
\and Keimei Gakuin Junior and Senior High School, 9-5-1 Yokoo, Suma-ku, Kobe, Japan\\
\email{runnerskg@gmail.com}\\
\and Osaka Metropolitan University, 3-3-138 Sugimoto, Sumiyoshi-ku, Osaka, Japan 
\and Waseda University, 513 Waseda-Tsurumaki-Cho, Sinjuku-ku, Tokyo, Japan 
\and Gifu University, 1-1 Yanagido,  Gifu, Japan \\
\email{suetsugu.koki@gmail.com} }
\maketitle              
\begin{abstract}
In this paper, we consider two-player impartial games with a pass-move. A disjunctive compound of games is a position in which, on each turn, the current player chooses one of the components and makes a legal move in it. 
\imp{For disjunctive compounds, it is known that the time to determine which player has a winning strategy is bounded by the time to compute the SG-values of the components plus the time for their XOR.} However, if we allow a pass-move during the play, the analysis of such games becomes much more difficult.
\imp{A pass-move allows each player to skip exactly one turn in non-terminal positions during the game, after which neither player may use a pass-move again.} We establish a homomorphism on the SG-values of games with a pass-move. That is, if every component satisfies a condition called one-move game, the SG-value of the disjunctive compound of the components with a pass-move is the same as the SG-value of \textsc{nim} with a pass-move where the size of every pile is the same as the SG-value of every component of the compound.
This guarantees that 
\imp{the time to determine which player has a winning strategy in a disjunctive compound with a pass can be bounded by the sum of the time to determine SG-values of all components  without a pass and a position in \textsc{nim} with a pass.} We also show how the homomorphism is used for determining SG-values of some \textsc{chocolate games}.

\end{abstract}

\section{Introduction}
\label{sec:intro}

In this paper, we consider two-player \emph{impartial games}, that is, two-player games in which, in every position, the sets of options for both players are the same. We assume that the games are under \emph{normal play convention}, that is, the player who has no legal move \imp{loses}.
Developing effective methods for determining whether a player has a winning strategy is a central topic in theoretical computer science. For example, \textsc{nim} is a two-player game with several piles of tokens, and in each turn, the current player chooses one of the piles and removes any positive number of tokens from the pile. Since we assume that games are under normal play convention, the player who has no legal move is the loser.
Bouton showed that in \textsc{nim}, the previous player has a winning strategy  if and only if the exclusive OR (XOR) of the numbers of tokens in the piles is zero \cite{bou}. This is  a much faster way to determine which player has a winning strategy than brute-force algorithms.

Furthermore, for any impartial games under normal play convention,  
Sprague and Grundy independently generalized Bouton's theory and showed that when a position is regarded as a disjunctive compound of some components, that is, in each turn, the player chooses exactly one of the components and makes a legal move on it, the previous player has a winning strategy if and only if the XOR of the parameter, called \emph{SG-value}, of each component is zero \cite{spr}, \cite{gru}. 
Therefore, if a position is a disjunctive compound, the time for determining which player has a winning strategy is bounded by the sum of the times for calculating SG-values of components and the time for calculating XOR of the SG-values.

On the other hand, it becomes more difficult to analyze games if we allow a pass-move in the play.
Under this convention, a pass-move may be used at most once in the game, but not when the position is terminal. Once the pass has been used by either player, it is no longer available. 
There are some early results considering this convention \cite{nimpass}, \cite{LC15}, \cite{integers1}, \cite{CLLW18}. Also, this convention is generalized as a sum of games called ``split sum'' in \cite{Hirsch20}.

\textsc{Nim} is a good example to explain how the pass-move makes it difficult to analyze games. 
\imp{As mentioned above, it is easy to solve three-pile \textsc{nim}.}
However, no mathematical formula is known for the previous player's winning position when a pass-move is allowed in this ruleset and this has been posed as an open problem in the list of unsolved problems in combinatorial game theory \cite{unsolve}.

In this paper, we consider disjunctive compounds with a pass-move and show that if the components satisfy a special property, called one-move game, the time \imp{to determine} which player has a winning strategy can be bounded by the sum of the time \imp{to determine SG-values of all components without a pass and a position in \textsc{nim} with a pass}, similar to the case of standard disjunctive compounds.

The outline of this paper is as follows.
In the latter part of this section, we introduce necessary early results for completeness. In Section \ref{sec:pass}, we establish a homomorphism of SG-values on disjunctive compound with a pass, which makes the calculation for determining which player has a winning strategy easier.
In Section \ref{sec:choco}, we show how the result can be used for analyzing games by using an example, \textsc{chocolate game}.
\subsection{Disjunctive compound}



We briefly review some of the necessary concepts of combinatorial game theory; refer to \cite{lesson}  for more details. Let $\mathbb{Z}_{\ge 0}$  be the set of nonnegative integers.


\begin{definition}\label{NPpositions} 
\imp{A position is referred to as a $\mathcal{P}$-{\em position} (resp. an $\mathcal{N}$-{\em position}) if  the previous (resp. next) player has a winning strategy in the position.}
\end{definition}

\begin{definition}\label{defofmexgrundy2}
\mbox{}
\begin{enumerate}
    
\item	For any position $g$ of an impartial game, there is a set of positions that can be reached in precisely one move from $g$, which is denoted by $\mathrm{move}(g)$. An element in $\mathrm{move}(g)$ is called an \emph{option} of $g$. 
Let $\mathrm{move}^0(g) = \{g\},$ and $\mathrm{move}^n(g) = \bigcup_{h\in \mathrm{move}^{n-1}(g)}\mathrm{move}(h)$ for any $n \ge 1$.
If $g' \in \mathrm{move}^n(g)$ for $n \ge 0$, $g'$ is a \emph{follower} of $g$. 
\item The \emph{minimum excluded value} $(\mathrm{mex})$ of a set $S$ of nonnegative integers is the least nonnegative integer that is not in $S$. That is, $\mathrm{mex}(S) = \min( \mathbb{Z}_{\ge 0} \setminus S)$. 
\item Each position $g$ of an impartial game has an associated \emph{Sprague-Grundy value}, or \emph{SG-value}, which is denoted by $\mathcal{G}(g)$.
	The SG-value is obtained recursively as follows. 
	$\mathcal{G}(g) = \mathrm{mex}(\{\mathcal{G}(h)\mid h \in \mathrm{move}(g)\}).$
\end{enumerate}
\end{definition}

\begin{theorem}[\cite{spr}, \cite{gru}]\label{theoremofsumg2}
For any position $g$ of an impartial game, 
	$\mathcal{G}(g) =0$ if and only if $g$ is a $\mathcal{P}$-position \imp{ and $\mathcal{G}(g) \neq 0$ if and only if $g$ is an $\mathcal{N}$-position.}
\end{theorem}



\begin{definition}
    Let $g_1, \ldots, g_n$ be positions in impartial games.
    A \emph{compound} of positions is a function $C$ such that
    $C(g_1, \ldots, g_n)$ is also a position in an impartial game.
\end{definition}


\begin{definition}
    Let $g_1, \ldots, g_n$ be positions in impartial games. The \emph{disjunctive compound} of $g_1, \ldots, g_n$, denoted by $C_+(g_1, \ldots, g_n)$, or, $g_1 + \cdots + g_n$, is a position such that the set of options of the position is 
    \begin{eqnarray*}
    \mathrm{move}(C_+(g_1, \ldots, g_n)) = \bigcup_{i=1}^n \{ C_+(g'_1,  \ldots, g'_n) \mid g'_i \in \mathrm{move}(g_i) \text{ and } g'_j = g_j \text{ for any } j \neq i\}.
    \end{eqnarray*}
\end{definition}

In other words, in a disjunctive compound of games, each player chooses exactly one component and makes a legal move on it. \imp{Disjunctive compound is one of the most famous compounds.}

Let $\oplus$ be the XOR operator for binary notation.

\begin{theorem}[\cite{spr}, \cite{gru}]
\label{thm:disjunctive}
For any positions $g_1, \ldots, g_n$ in impartial games, \imp{we have}
    \begin{eqnarray*}
    \mathcal{G}(C_+(g_1, \ldots, g_n)) = \mathcal{G}(g_1) \oplus \cdots \oplus \mathcal{G}(g_n).
    \end{eqnarray*}
\end{theorem}

For any compound $C$, let $C^\mathrm{nim}(m_1, \ldots, m_n)$ be a position such that every component $g_i$ of $C$ is a single-pile \textsc{nim} with $m_i$ tokens.
For example, since $n$-pile \textsc{nim} is disjunctive compound of positions of single-pile \textsc{nim}, a position in $n$-pile \textsc{nim} can be denoted by  $C_+^\mathrm{nim}(m_1, \ldots, m_n)$.

The SG-value of $n$-pile \textsc{nim} is equal to the XOR of the numbers of tokens in the piles since the SG-value of one-pile \textsc{nim} with $m$ tokens is $m$. Thus, we have
$
\mathcal{G}(C_+^\mathrm{nim}(m_1, \ldots, m_n)) = m_1 \oplus \cdots \oplus m_n
$
and from Theorem \ref{thm:disjunctive}, we have a homomorphism
\begin{eqnarray*}
\mathcal{G}(C_+(g_1, \ldots, g_n)) = \mathcal{G}(C_+^\mathrm{nim} (\mathcal{G}(g_1), \ldots, \mathcal{G}(g_n))).
\end{eqnarray*}

When we consider compounds other than disjunctive compound, SG-value sometimes does not work. 
Therefore, it is important to reveal that for what compound $C$ and what kind of positions $g_i$, the homomorphism 
\begin{eqnarray*}
\mathcal{G}(C(g_1, \ldots, g_n)) = \mathcal{G}(C^\mathrm{nim} (\mathcal{G}(g_1), \ldots, \mathcal{G}(g_n)))
\end{eqnarray*}
holds. We call this homomorphism an \emph{SG-homomorphism}.






\subsection{Hypergraph compound and SG-decreasing games}

In \cite{hyper}, the \emph{hypergraph compound} of games is defined as follows.

Let $\mathcal{H}$ be a hypergraph $\mathcal{H} \subseteq 2^{[n]} \setminus \{ \emptyset \}$, where $[n] = \{1, 2, \ldots, n\}$. On each vertex $i$ of $\mathcal{H}$, there is an impartial position $g_i$. A player, on their turn, chooses a hyperedge $H \in \mathcal{H}$ and makes a move in every position $g_i$, where $i \in H$.  

The hypergraph compound of $g_1, \ldots, g_n$ is denoted by $C_{\mathcal{H}}(g_1, \ldots, g_n)$ for given hypergraph $\mathcal{H}$.



\begin{definition}
    Let $g$ be a position in an impartial game. $g$ is an \emph{SG-decreasing} game if for any nonnegative integer $n$, every $g' \in \mathrm{move}^n(g)$ satisfies $\mathcal{G}(g') > \mathcal{G}(g'')$ for any $g'' \in \mathrm{move}(g')$.
\end{definition}

Positions in one-pile \textsc{nim}, \textsc{minimal nim} \cite{minnim}, and \textsc{exact-$k$ nim} \cite{exact} with $2k > n$, are examples of SG-decreasing games.

\begin{theorem}[\cite{hyper}]
    Let $\mathcal{H}$ be a hypergraph $\mathcal{H} \subseteq 2^{[n]} \setminus \{ \emptyset \}$. Assume that $g_1, \ldots, g_n$ are SG-decreasing games. Then, \imp{we have} 
    \begin{eqnarray*}
\mathcal{G}(C_{\mathcal{H}}(g_1, \ldots, g_n)) = \mathcal{G}(C^\mathrm{nim}_{\mathcal{H}}(\mathcal{G}(g_1), \ldots, \mathcal{G}(g_n))).
\end{eqnarray*}
\end{theorem}


Table \ref{tab:homomorphism} summarizes what compound and what kind of positions satisfy SG-homomorphism. In addition to the shown two pairs, we introduce one-pass compound and one-move games in the next section, as the third example of SG-homomorphism. Actually, the set of one-move games is a subset of impartial games, and the set of SG-decreasing games is a subset of one-move games. Thus, we consider that our result falls between the two results.

\begin{table}[tb]
 \caption{Pairs of compound and positions which satisfy SG-homomorphism.}
    \label{tab:homomorphism}

    \centering
    \begin{tabular}{c|ccc}
        Reference & Sprague \cite{spr} and Grundy \cite{gru}& This paper & Boros et al. \cite{hyper} \\ \hline
        Compound & disjunctive compound & one-pass compound & hypergraph compound \\ 
        Positions & any positions in impartial games & one-move games & SG-decreasing games         
    \end{tabular}
   \end{table}
\subsection{\textsc{Chocolate games}}



\textsc{Chocolate games} were first presented in \cite{robin}. \textsc{Chocolate games} look similar to 
 the \textsc{chomp} game that was presented in \cite{gale}, but how to cut these chocolate bars differs from the cutting rule in \textsc{chomp} game. Therefore, \textsc{chocolate games} differ considerably from \textsc{chomp} games. Given that the classical three-pile \textsc{nim} is mathematically equivalent to a rectangular \textsc{chocolate game}, \textsc{chocolate games} are generalizations of the classical \textsc{nim}.


\imp{A two- or three-dimensional chocolate bar is a grid of squares or cubes containing at least one bitter cell printed in black. See Figures \ref{two2dchoco} and \ref{two3dchoco}. Players alternately break the bar along a groove and take the separated piece. The player who leaves the opponent with only the bitter cell wins. Some examples of cut in three-dimensional chocolate bar are shown in Figures. \ref{vcut1}, \ref{vcut2}, and \ref{hcut1}.}

\begin{figure}[tb]
\begin{tabular}{cc}
\begin{minipage}{.53\textwidth}
\centering
\includegraphics[height=1.4cm]{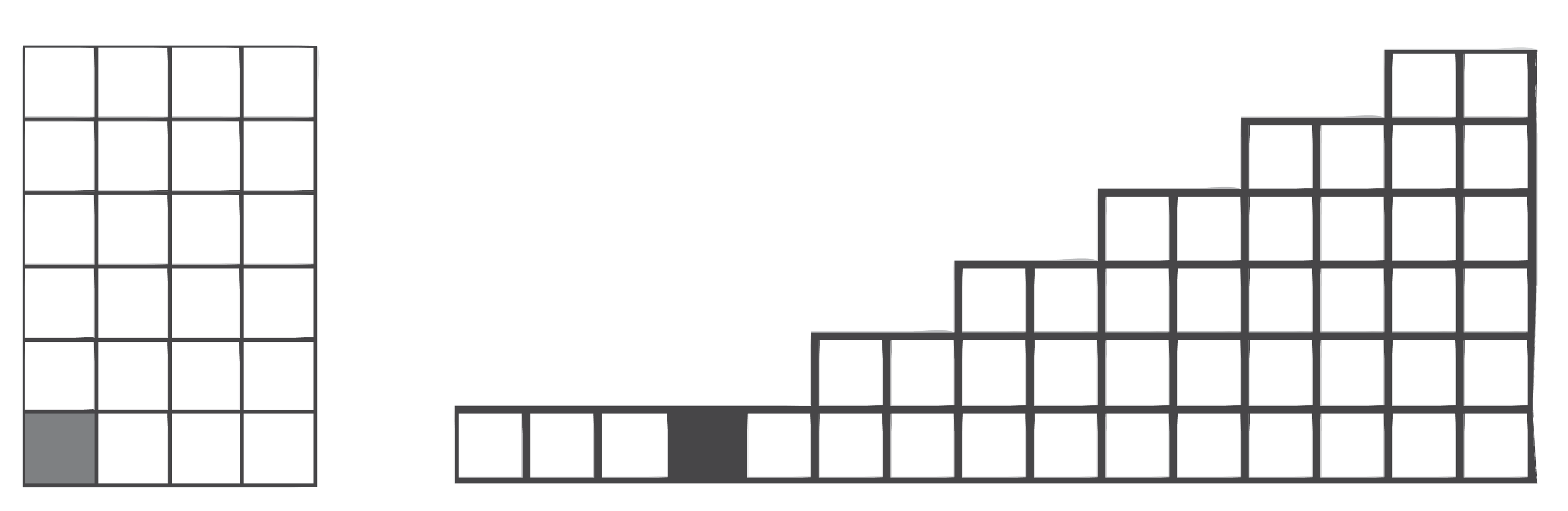}
\caption{Two-dimensional chocolate bars.}\label{two2dchoco}
\end{minipage}
\begin{minipage}{.42\textwidth}
\centering
\includegraphics[height=1.6cm]{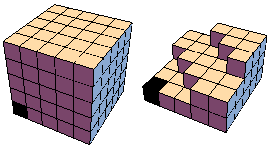}
\caption{Three-dimensional chocolate bars.}\label{two3dchoco}
\end{minipage}
\end{tabular}
\end{figure}

\begin{figure}[tb]
\begin{tabular}{ccc}
\begin{minipage}{.31\textwidth}
\centering
\includegraphics[height=2.0cm]{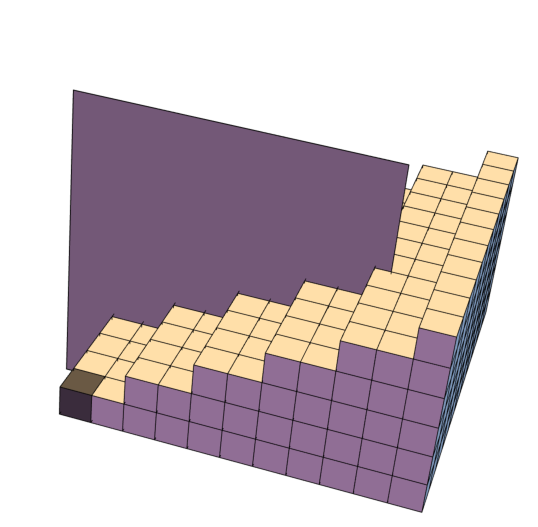}
\caption{A vertical cut.}\label{vcut1}
\end{minipage}
\begin{minipage}{.26\textwidth}
\centering
\includegraphics[height=2.0cm]{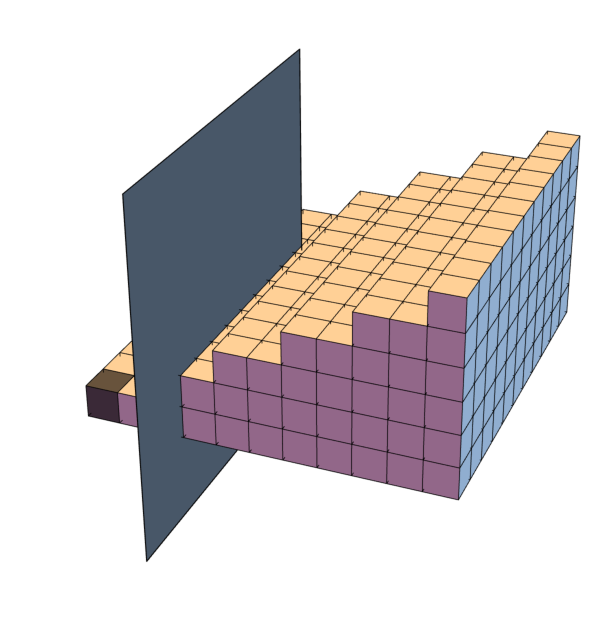}
\caption{A vertical cut.}\label{vcut2}
\end{minipage}
\begin{minipage}{.31\textwidth}
\centering
\includegraphics[height=1.2cm]{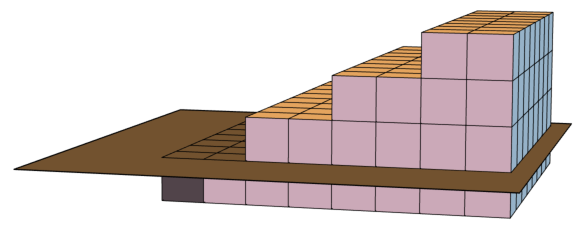}
\vspace{0.9cm}
\caption{A horizontal cut.}\label{hcut1}
\end{minipage}
\end{tabular}
\end{figure}


\begin{definition}\label{definitionoffunctionf3d}
\mbox{}
A single-variable function $f$ of $\mathbb{Z}_{\geq0}$ onto itself is said to be a \emph{monotonically increasing} if
$f(u) \leq f(v)$ for $u,v \in \mathbb{Z}_{\geq0}$, with $u \leq v$, \imp{and} 
a two-variable function $F: \mathbb{Z}_{\geq0}\times \mathbb{Z}_{\geq0} \rightarrow \mathbb{Z}_{\geq0}$ is said to be a \emph{monotonically increasing} if $F(u,v) \leq F(x,z)$ for $x,z,u,v \in \mathbb{Z}_{\geq0}$ with $u \leq x$ and $v \leq z$.
\end{definition} 


\begin{definition}\label{defofbarwithfunc3d}
Let $F$ be a monotonically increasing two-variable function. Let $x,y,z \in \mathbb{Z}_{\geq0}$ such that $y \leq F(x,z)$.
The three-dimensional chocolate bar comprises a set of $1 \times 1 \times 1$ boxes. 
For $u,w \in \mathbb{Z}_{\geq0}$ such that $u \leq x$ and $w \leq z$, the height of the column at position $(u,w)$ is $ \min (F(u,w),y) +1$. There is a bitter box at position $(0,0)$.
We denote this chocolate bar by $CB(F,x,y,z)$. Note that $x+1, y+1$, and $z+1$ are the length, height, and width of the bar, respectively.
\end{definition}

Next, we define $\mathrm{move}_F(x, y, z)$ in Definition \ref{movefor3dimension}. $\mathrm{move}_F(x, y, z)$ is a set that contains all the positions that can be directly reached from position $CB(F,x, y, z)$ in a single step. That is, $(x', y', z') \in \mathrm{move}_F(x, y, z)$ if and only if $CB(F, x', y', z') \in \mathrm{move}(CB(F, x, y, z))$.

\begin{definition}\label{movefor3dimension}
	For $x,y,z \in \mathbb{Z}_{\ge 0}$, 
\begin{displaymath}
 \mathrm{move}_F(x,y,z)=\{(u,\min(F(u,z),y),z ):u<x \} \cup \{(x,v,z ):v<y \}   \nonumber    
\end{displaymath}
\begin{displaymath}
\cup \{ (x,\min(y, F(x,w) ),w ):w<z \},\nonumber  
\end{displaymath}
where $u,v,w \in \mathbb{Z}_{\ge 0}$.
\end{definition}
The following examples show how the coordinates change when we cut chocolate bars. Figure \ref{coordinate3d} shows which direction corresponds to $x, y,$ or $z$. Let $\lfloor \ \rfloor$ be the floor function.
\begin{example} 
 Let $F_1(x,z)=$ $ \lfloor \frac{x+z}{3}\rfloor$. \imp{$CB(F_1, 7, 6, 13)$ and $CB(F_1, 7, 3, 4)$, shown in Figures \ref{f1367} and \ref{f437}, are positions of this \textsc{chocolate game} and $CB(F_1, 7, 3, 4)$ is an option of $CB(F_1, 7, 6, 13)$.}
\begin{figure}[tb]
\begin{tabular}{ccc}
\begin{minipage}{.3\textwidth}
	\centering
\includegraphics[height=1.8cm]{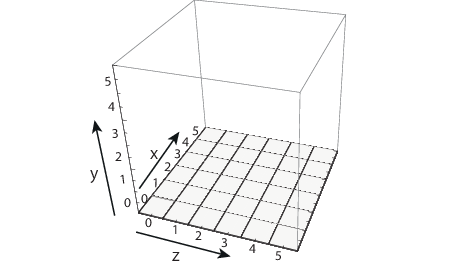}
  \caption{Coordinate.}\label{coordinate3d}
	\end{minipage}
\begin{minipage}{.33\textwidth}
	\centering
\includegraphics[height=1.8cm]{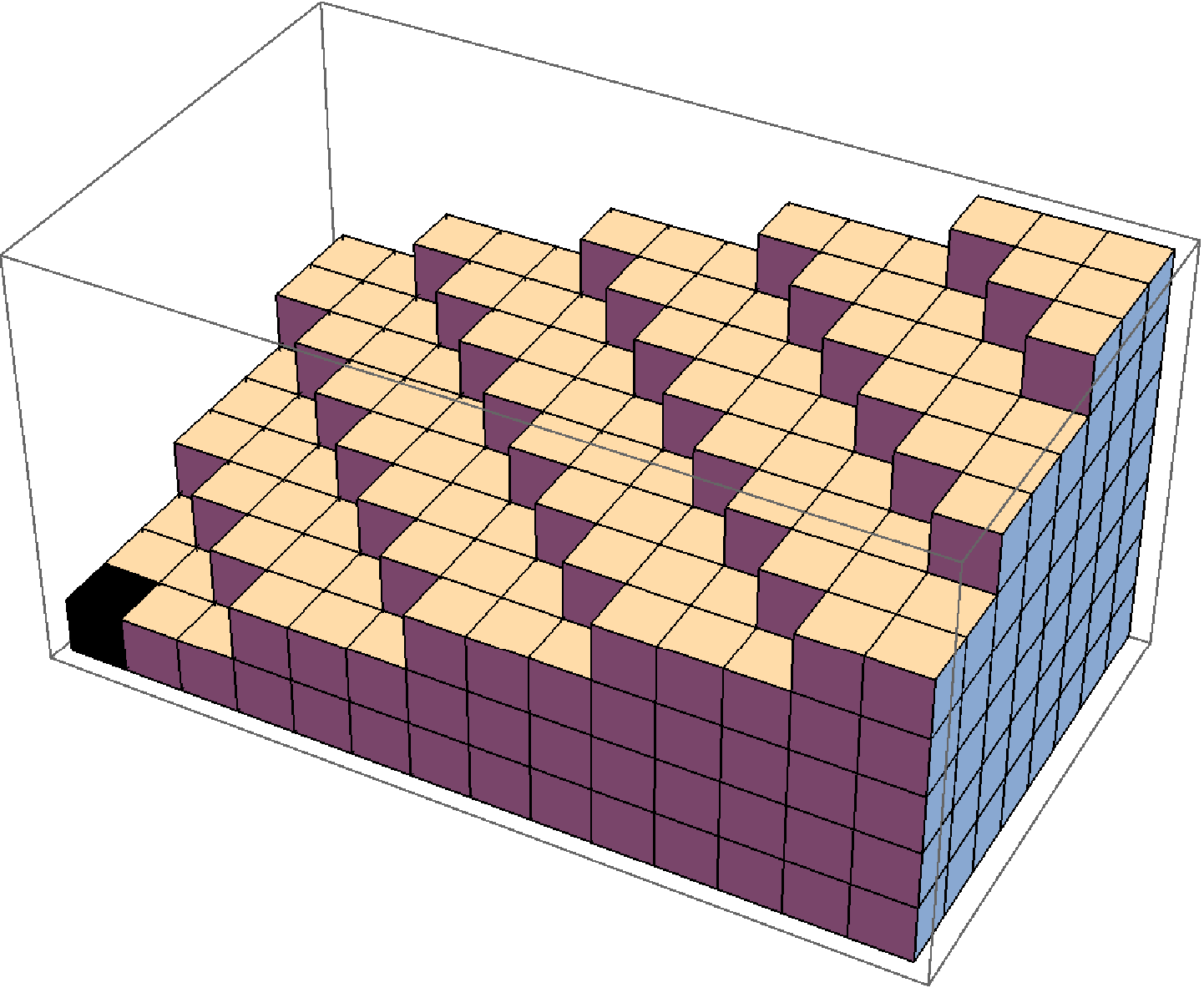}
  \caption{ $CB(F_1,7,6,13)$}\label{f1367}
\end{minipage}
\begin{minipage}{.33\textwidth}
		\centering
\includegraphics[height=1.8cm]{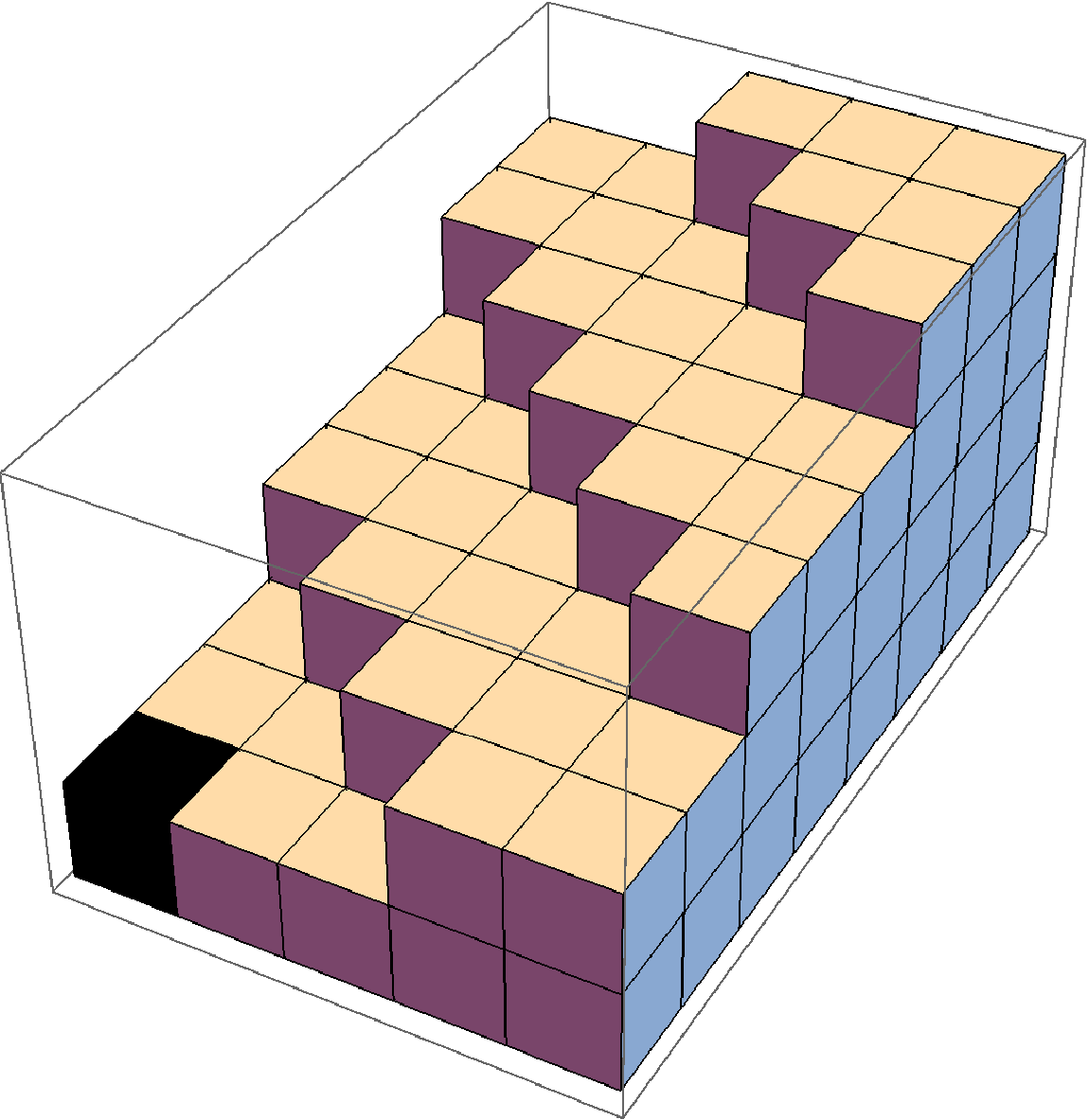}
 \caption{ $CB(F_1,7,3,4)$ }\label{f437}
\end{minipage}	
\end{tabular}
\end{figure}
\end{example}
\begin{example}\label{howtocutchoco} 
 Let $F_2(x,z)$ $= \lfloor \frac{z}{2}\rfloor$. \imp{Positions $CB(F_2, 3, 3, 7), CB(F_2, 5, 2, 7), $ and $CB(F_2, 5, 2 ,5)$, shown in Figures \ref{3D337}, \ref{3D527}, and \ref{3Dchocolate}, are options of $CB(F_2, 5, 3, 7)$, shown in Figure \ref{3Dchocolate22}.}
\imp{Note that there are three different directions of cutting chocolates and in some cases, we reduce two coordinates at the same time.}
\end{example}

\begin{figure}[tb]
\begin{tabular}{cccc}
		\begin{minipage}{.24\textwidth}
		\centering
\includegraphics[height=1.8cm]{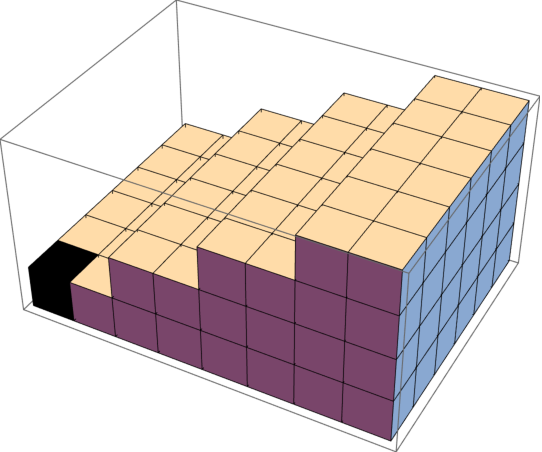}
 \caption{$CB(F_2,5,3,7)$.}\label{3Dchocolate22}
	\end{minipage}
		\begin{minipage}{.24\textwidth}
		\centering
\includegraphics[height=1.8cm]{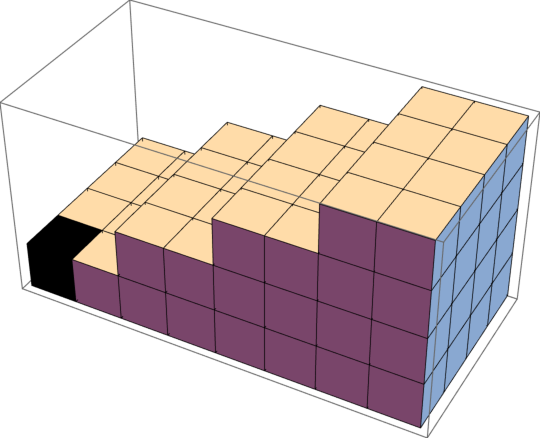}
 \caption{$CB(F_2,3,3,7)$.}\label{3D337}
\end{minipage}
		\begin{minipage}{.24\textwidth}
		\centering
\includegraphics[height=1.8cm]{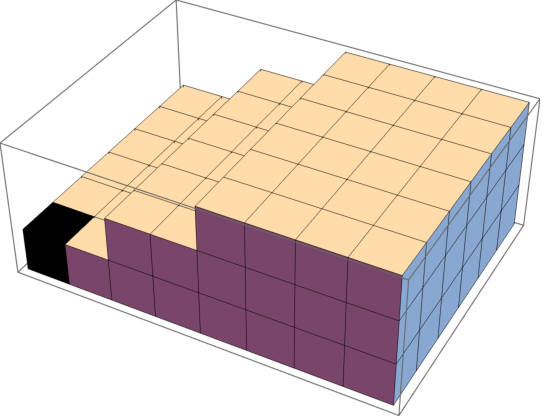}
 \caption{$CB(F_2,5,2,7)$.}\label{3D527}
	\end{minipage}
		\begin{minipage}{.24\textwidth}
		\centering
\includegraphics[height=1.8cm]{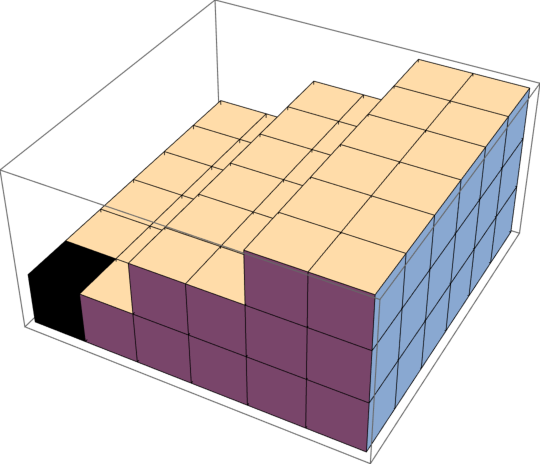}
 \caption{$CB(F_2,5,2,5)$.}\label{3Dchocolate}
\end{minipage}
	\end{tabular}
\end{figure}

\imp{The original two-dimensional chocolate bar shown at left in Figure \ref{two2dchoco} and three-dimensional chocolate bar shown at left in Figure \ref{two3dchoco} were introduced by Robin \cite{robin}. Since the length, height, and width can be reduced independently,  $m \times n \times \ell$ cuboid chocolate bar can be considered as three-pile \textsc{nim} whose piles have $m-1$, $n-1$ and $\ell - 1$ tokens. Since the SG-value of the \textsc{nim} is $(m - 1) \oplus (n - 1) \oplus (\ell - 1)$, the SG-value of  $m \times n \times \ell$ cuboid chocolate bar is $(m - 1) \oplus (n - 1) \oplus (\ell - 1)$.}
Here, the following question arises naturally: 
What are the necessary and sufficient conditions for a three-dimensional chocolate bar to have the SG-value $(m-1) \oplus (n-1) \oplus (\ell-1)$, where $m, n$, and $\ell$ are the length, height, and width of the bar, respectively?
The answer to this question is presented in \cite{integer2021}, which is referred to as Theorem \ref{sufficientcond3d} in this section.

\begin{definition}\label{definitionoffunctionf}
\mbox{}
\begin{itemize}    
\item[(i)] Let $h(z)$ be a single-variable monotonically increasing function.
The function $h$ is said to have the NS-property if $h$ satisfies $h(0)=0$ and the following condition: \\
	Suppose that 
	$\lfloor \frac{z}{2^i}\rfloor = \lfloor \frac{z^{\prime}}{2^i}\rfloor$ 
	for some $z $, $ z^{\prime} \in \mathbb{Z}_{\geq 0}$, and some positive integer $i$.
	Then, 
		$\lfloor \frac{h(z)}{2^{i-1}} \rfloor = \lfloor \frac{h(z^{\prime})}{2^{i-1}}\rfloor.$ 
\item[(ii)] Let $F(x,z)$ be a double-variable monotonically increasing function.
Let $g_n(z) = F(n,z)$ and $h_m(x) =F(x,m)$ for $n,m \in \mathbb{Z}_{\geq 0}$.
The function $F$ is said to have the NS-property if $g_n$ and $h_m$ satisfy the NS-property in $(i)$ of this definition.
\end{itemize}
\end{definition}

\imp{The following proposition shows}
 examples of functions that satisfy the NS-property in $(i)$ of Definition \ref{definitionoffunctionf}. \imp{For the proofs of the proposition, see the appendix.} 

\begin{proposition}\label{lemmaforfloorzbyk} \mbox{}
\begin{enumerate}
\item Let $h(z)=\lfloor \frac{z}{2k}\rfloor$ for some positive integer $k$. Then 
$h(z)$ satisfies NS-property.
\item  Let $h(z)=2^{\lfloor \log_2z \rfloor}-1$ for $z > 0$ and 
$h(0)=0.$
Then, $h(z)$ satisfies NS-property.
\end{enumerate}
\end{proposition}


\begin{lemma}
\label{lem:hbound}
Let $h(z)$ be a function which has NS-property. For any positive integer $z$, 
we have $h(z) \leq 2^{\lfloor \log_2 z \rfloor} - 1$ .

\end{lemma}
\begin{proof}

    Let $i = \lfloor \log_2 z \rfloor$. We have $\lfloor \frac{0}{2^{i+1}} \rfloor = \lfloor \frac{z}{2^{i+1}}  \rfloor = 0$. Therefore, from the definition of NS-property, $\lfloor  \frac{h(z)}{2^{i}} \rfloor = \lfloor \frac{h(0)}{2^{i}}\rfloor = 0$, which means $h(z) \leq 2^{i} - 1$.
    \qed
\end{proof}

\begin{theorem}\label{sufficientcond3d}
Let $F(x,z)$ be a monotonically increasing function.
$F(x,z)$ has  NS-property 
if and only if the SG-value of chocolate bar $CB(F,x,y,z)$ is $x \oplus y \oplus z$.
\end{theorem}
\imp{This result is shown in \cite{integer2021}.}
By Theorem \ref{sufficientcond3d}, the formula for SG-values in \textsc{chocolate games} $CB(F,x,y,z)$ with NS-property is the same as that of three-pile \textsc{nim}, but as shown in Section \ref{sec:choco}, some \textsc{chocolate games} with NS-property also have a formula for 
$\mathcal{P}$-positions even if a pass-move is allowed.
Therefore, the pass-move was found to have a minimal impact on these \textsc{chocolate games}. This is remarkable since three-pile \textsc{nim} does not have any known formula for $\mathcal{P}$-positions when a pass-move is allowed.

\begin{corollary}
    Let $h(z)$ be a monotonically increasing function. Assume that function $h(z)$ has  NS-property.
    Then, consider two-dimensional \textsc{chocolate game} $CB_2(h, y, z) = CB(F, 0, y, z)$, where $F(x, z) = h(z)$ and $y$ is restricted by $F$ as $y \leq F(x, z)$.
    The SG-value of the chocolate bar $CB_2(h, y, z)$ is $y \oplus z$. 
\end{corollary}

\begin{proof}
    When $h(z)$ has NS-property, $F(x, z) = h(z)$ also has NS-property.
    Therefore, the SG-value of position $CB_2(h, y, z)$ is the same as $CB(F, 0, y, z)$, which is $0 \oplus y \oplus z = y \oplus z$. \qed
\end{proof}

Table \ref{tab:NS} shows the SG-value of \textsc{chocolate game} $CB_2(h, y, z)$ for small bars, where $h$ has NS-property. Note that from Lemma \ref{lem:hbound}, whatever the function $h$ is, we do not need to consider the case $y > 2^{\lfloor \log_2 z \rfloor} -1$. Therefore, the shape of the  chocolate bar of a position where $z \leq 15$ must be included in this table.

\begin{table}[tb]
    \caption{SG-values of \textsc{chocolate game} $CB_2(h, y, z)$}
    \label{tab:NS}

    \centering
    \begin{tabular}{c|cccccccccccccccc}
        $y \backslash z$ & 0 & 1 & 2 & 3 & 4 & 5 & 6 & 7 & 8 & 9 & 10 & 11 & 12 & 13 & 14 & 15 \\ \hline
        0 & 0 & 1 & 2 & 3 & 4 & 5 & 6 & 7 & 8 & 9 & 10 & 11 & 12 & 13 & 14 & 15 \\
        1 &   &   & 3 & 2 & 5 & 4 & 7 & 6 & 9 & 8 & 11 & 10 & 13 & 12 & 15 & 14 \\
        2 &   &   &   &   & 6 & 7 & 4 & 5 & 10 & 11 & 8 & 9 & 14 & 15 & 13 & 12\\
        3 &   &   &   &   & 7 & 6 & 5 & 4 & 11 & 10 & 9 & 8 & 15 & 14 & 13 & 12 \\
        4 &   &   &   &   &   &   &   &   & 12 & 13 & 14 & 15 & 8 & 9 & 10 & 11 \\
        5 &   &   &   &   &   &   &   &   & 13 & 12 & 15 & 14 & 9 & 8 & 11 & 10 \\
        6 &   &   &   &   &   &   &   &   & 14 & 15 & 12 & 13 & 10 & 11 & 8 & 9 \\
        7 &  &  &   &   &   &   &   &   &  15 & 14  & 13  & 12  & 11 & 10 & 9 & 8 \\
        
    \end{tabular}
\end{table}

\begin{lemma}
\label{lem:16}
    Assume that $h(z)$ has NS-property.
    For any $z \geq 16,$  we have $ y \oplus z \geq 16$ if $y \leq h(z)$.
\end{lemma}
\begin{proof}
    Assume that $h(z)$ has NS-property.
    From Lemma \ref{lem:hbound}, for any integer $z$, there is an integer $i$ such that $2^i \leq z$ and $h(z) < 2^i$. 
Thus, when $z \geq 16$, we have $i \geq 4$ and $y \oplus z  \geq 2^4=16$. \qed
    
\end{proof}
\begin{lemma}
\label{lem:leq8}
\imp{Assume that $h$ has NS-property. Let $y$ and $z$ be nonnegative integers with $y \le h(z)$. Then, the following holds.}

\begin{itemize}
\item 
$CB_2(h, y, z)$ has SG-value $0$ if and only if $y = z = 0$.

\item 
$CB_2(h, y, z)$ has SG-value $1$ if and only if $y = 0$ and $z = 1$.

\item 
$CB_2(h, y, z)$ has SG-value $2$ if and only if $(y, z) \in \{(0, 2), (1, 3)\}$.

\item 
$CB_2(h, y, z)$ has SG-value $3$ if and only if $(y, z) \in \{(0, 3), (1, 2)\}$.

\item $CB_2(h, y, z)$ has SG-value $4$ if and only if $(y, z) \in \{(0, 4), (1, 5), (2, 6), (3, 7)\}$.

\item $CB_2(h, y, z)$ has SG-value $5$ if and only if $(y, z) \in \{(0, 5), (1, 4), (2, 7), (3, 6)\}$.

\item $CB_2(h, y, z)$ has SG-value $6$ if and only if $(y, z) \in \{(0, 6), (1, 7), (2, 4), (3, 5)\}$.

\item $CB_2(h, y, z)$ has SG-value $7$ if and only if $(y, z) \in \{(0, 7), (1, 6), (2, 5), (3, 4)\}$.

\item $CB_2(h, y, z)$ has SG-value $8$ if and only if $(y, z) \in \{(0, 8), (1, 9), (2, 10), (3, 11), \allowbreak (4,12), \allowbreak (5, 13), (6,14), (7,15)\}$.
\end{itemize}

\end{lemma}
\begin{proof}
    From Lemma \ref{lem:16}, we only need to consider the case $z \leq 15.$ Therefore, from Table \ref{tab:NS}, we can confirm that the statement is true. \qed
\end{proof}

\section{One-move games with a pass}
\label{sec:pass}
In this section, we show that the SG-value of a disjunctive compound of games with a pass can be calculated by using SG-value of \textsc{nim} with a pass if every component is a game that can arrive at a terminal position in one move.

\begin{definition}
\label{def:dispas}
Let $g_1, \ldots, g_n$ be positions in impartial games.
{\em One-pass compound} of $g_1, \ldots, g_n$ is denoted by $C_{(+, \mathrm{pass})}(g_1, \ldots, g_n)$, and the set of its options is  
\begin{eqnarray*}
\bigcup_{i=1}^n \{(C_{(+, \mathrm{pass})}(g'_1, \ldots, g'_n)) \mid g'_i \in {\rm move}(g_i) \text{ and } g'_j = g_j \text{ for any } j \neq i\} \cup \{C_+(g_1, \ldots, g_n)\}, 
\end{eqnarray*}
when  at least one of  $g_i$ is not a terminal position, and if every $g_i$ is a terminal position, then 
$\mathrm{move}(C_{(+, \mathrm{pass})}(g_1, \ldots, g_n)) =\emptyset.$

\end{definition}

That is, one-pass compound is almost the same as disjunctive compound, but the players can make a pass-move except in a terminal position, and once the pass-move is used, neither player may use it again. Note that when the original position is a $\mathcal{P}$-position, the position changes to an $\mathcal{N}$-position when a pass-move is added since the player can use the pass. However, when the original position is $\mathcal{N}$-position, there are two cases that it changes to a $\mathcal{P}$-position or it remains to be an $\mathcal{N}$-position. Therefore, even if the winner of the original position is determined, it is still hard to determine the winner if a pass-move is added.

\begin{definition}
    A position $g$ is \emph{one-move game} if its follower has SG-value $0$ if and only if the follower is a terminal position. 
\end{definition}

Note that if a \imp{position} is SG-decreasing, then it is also a one-move game, \imp{but} the converse may not hold. For example, two-dimensional \textsc{chocolate game} where function $h$ satisfies NS-property is a one-move game, \imp{but} like the move $CB_2(h, 1, 3)$ to $CB_2(h, 1, 2)$, there are moves that increase SG-values \imp{(see Table \ref{tab:NS})}.
We also note that if a one-move game is a single component, then the game will end in at most one move since if it does not have SG-value $0$, then it must have an option to a position whose SG-value is $0$, that is, a terminal position. However, if two or more one-move games are combined by disjunctive compound, then it is not guaranteed to end in a certain number of \imp{moves}, since moving \imp{a component} to \imp{a} terminal position may not be a winning strategy.

\begin{theorem}
\label{thm:pass}
Let $g_1, \ldots, g_n$ be one-move games.
Then, \imp{we have}
    \begin{eqnarray*}        
\mathcal{G}(C_{(+, \mathrm{pass})}(g_1, \ldots, g_n)) = \mathcal{G}(C^\mathrm{nim}_{(+, \mathrm{pass})}(\mathcal{G}(g_1), \ldots, \mathcal{G}(g_n))).
    \end{eqnarray*}


%

\end{theorem}
\begin{proof}

    We prove this by induction on the sum of the heights of game trees of $g_1, \ldots, g_n$. When every $g_i$ is terminal position, 
    \begin{eqnarray*}
\mathcal{G}(C_{(+, \mathrm{pass})}(g_1, \ldots, g_n)) = \mathcal{G}(C^\mathrm{nim}_{(+, \mathrm{pass})}(\mathcal{G}(g_1), \ldots, \mathcal{G}(g_n))) = 0
\end{eqnarray*}
    Assume that at least one of $g_i$ is not  a terminal position. Then, from the induction hypothesis,
    \begin{eqnarray*}
    &&\mathcal{G}(C_{(+, \mathrm{pass})}(g_1, \ldots, g_n)) \\ &=& \mathrm{mex}
    (\bigcup_{i=1}^n\{\mathcal{G}(C_{(+, \mathrm{pass})}(g'_1, \ldots, g'_n)) \mid g'_i \in {\rm move}(g_i) \text{ and } g'_j = g_j \text{ for any } j \neq i\}  \cup  \{\mathcal{G}(C_{+}(g_1, \ldots, g_n)) \}) \\ 
    &=& \mathrm{mex}(\bigcup_{i=1}^n \mathcal{G}(C^\mathrm{nim}_{(+, \mathrm{pass})}(\mathcal{G}(g'_1), \ldots, \mathcal{G}(g'_n)))\mid g'_i \in {\rm move}(g_i) \text{ and } g'_j = g_j \text{ for any } j \neq i\}     \cup  \{\mathcal{G}(g_1) \oplus \cdots \oplus \mathcal{G}(g_n)\}). \\
    \end{eqnarray*}
    \imp{Let $S_1 = \bigcup_{i=1}^n \mathcal{G}(C^\mathrm{nim}_{(+, \mathrm{pass})}(\mathcal{G}(g'_1), \ldots, \mathcal{G}(g'_n)))\mid g'_i \in {\rm move}(g_i) \text{ and } g'_j = g_j \text{ for any } j \neq i\}$ and $S_2 = \{\mathcal{G}(g_1) \oplus \cdots \oplus \mathcal{G}(g_n)\}$. In order to prove $\mathcal{G}(C^\mathrm{nim}_{(+, \mathrm{pass})}(\mathcal{G}(g_1), \ldots, \mathcal{G}(g_n))) = {\rm mex}(S_1 \cup S_2),$ we show that $\mathcal{G}(C^\mathrm{nim}_{(+, \mathrm{pass})}(\mathcal{G}(g_1), \ldots, \mathcal{G}(g_n))) \not \in S_1 \cup S_2$ and for any $\ell < \mathcal{G}(C^\mathrm{nim}_{(+, \mathrm{pass})}(\mathcal{G}(g_1), \ldots, \mathcal{G}(g_n))), \ell \in S_1 \cup S_2$. }
    Here, 
    \begin{eqnarray*}
    &&\mathcal{G}(C^\mathrm{nim}_{(+, \mathrm{pass})}(\mathcal{G}(g_1), \ldots, \mathcal{G}(g_n))) \not \in S_1
    \end{eqnarray*}
    because if $m_i \neq m'_i$, $$\mathcal{G}(C^\mathrm{nim}_{(+, \mathrm{pass})}(m_1, \ldots, m_{i-1}, m_i, m_{i+1}, \ldots, m_n)) \neq \mathcal{G}(C^\mathrm{nim}_{(+, \mathrm{pass})}(m_1, \ldots, m_{i-1}, m'_i, m_{i+1}, \ldots, m_n)).$$  Furthermore, since every $g_i$ is a one-move game and at least one of $g_i$ is not a terminal position, at least one of $\mathcal{G}(g_i)$ is a positive integer and therefore, $$\mathcal{G}(C^\mathrm{nim}_{(+, \mathrm{pass})}(\mathcal{G}(g_1),\ldots, \mathcal{G}(g_n)) \neq \mathcal{G}(g_1)\oplus \ldots \oplus  \mathcal{G}(g_n) = \mathcal{G}(C_+^\mathrm{nim}(\mathcal{G}(g_1), \ldots, \mathcal{G}(g_n))).$$

   Next, assume that $\ell < \mathcal{G}(C^\mathrm{nim}_{(+, \mathrm{pass})}(\mathcal{G}(g_1), \ldots, \mathcal{G}(g_n)))$. Since $$\mathcal{G}(C^\mathrm{nim}_{(+, \mathrm{pass})}(\mathcal{G}(g_1), \ldots, \mathcal{G}(g_n))) = \mathrm{mex}(\{\mathcal{G}(G) \mid G \in \mathrm{move}(C^\mathrm{nim}_{(+, \mathrm{pass})}(\mathcal{G}(g_1), \ldots, \mathcal{G}(g_n)))\}),$$ there is an integer $i$ such that $\ell = \mathcal{G}(C^\mathrm{nim}_{(+, \mathrm{pass})}(\mathcal{G}(g_1),\ldots, \mathcal{G}(g_{i-1}), k, \mathcal{G}(g_{i+1}),\ldots, \mathcal{G}(g_n))),$ where $k < \mathcal{G}(g_i)$, or $\ell = \mathcal{G}(C^\mathrm{nim}_+(\mathcal{G}(g_1),  \ldots, \mathcal{G}(g_n))) $. Then, for the former case, 
   $\ell \in S_1$
   and for the latter case, $\ell  = \mathcal{G}(g_1) \oplus \cdots \oplus \mathcal{G}(g_n).$ Thus, $\ell \in S_1 \cup S_2$
   and from \imp{the definition of ${\rm mex}$}, 
   $$
   \mathcal{G}(C_{(+, \mathrm{pass})}^\mathrm{nim}(\mathcal{G}(g_1), \ldots, \mathcal{G}(g_n))) = {\rm mex}(S_1 \cup S_2)
   $$
   holds, so we have
   \begin{eqnarray*}
   \mathcal{G}(C_{(+, \mathrm{pass})}(g_1, \ldots, g_n)) = \mathcal{G}(C^\mathrm{nim}_{(+, \mathrm{pass})}(\mathcal{G}(g_1), \ldots, \mathcal{G}(g_n))). 
   \end{eqnarray*} 
   \qed
\end{proof}
Note that when at least one of $g_i$ is not a one-move game, it is not guaranteed that SG-homomorphism  holds. Let us assume that one of the components, say, $g_1$, is not a one-move game and $g_2, \ldots, g_n$ are one-move games. Then, $g_1$ has a follower $g'_1$ whose SG-value is zero but not a terminal position.
Let $g'_2, \ldots, g'_n$ be terminal positions which can be reached from $g_2, \ldots, g_n$ in one move. Then, 
\begin{eqnarray*}
\mathcal{G}(C^\mathrm{nim}_{(+, \mathrm{pass})}(\mathcal{G}(g'_1), \mathcal{G}(g'_2) \ldots, \mathcal{G}(g'_n))) = \mathcal{G}(C^\mathrm{nim}_{(+, \mathrm{pass})}(0, 0, \ldots, 0)) = 0,
\end{eqnarray*}
but
$\mathcal{G}(C_{(+, \mathrm{pass})}(g'_1, g'_2 \ldots, g'_n)) \neq 0$
since $C_{(+, \mathrm{pass})}(g'_1, g'_2 \ldots, g'_n)$ has an option $C_+(g'_1, g'_2 \ldots, g'_n),$ whose SG-value is
\begin{eqnarray*}
\mathcal{G}(C_+(g'_1, g'_2 \ldots, g'_n) ) = \mathcal{G}(g'_1) \oplus \mathcal{G}(g'_2) \oplus \cdots \oplus \mathcal{G}(g'_n) = 0\oplus 0\oplus \cdots \oplus 0  = 0.
\end{eqnarray*}


From this theorem, the time needed to compute the SG-value of a one-pass compound of one-move games is bounded by
the total time required to compute the SG-values of all the component games, plus
the time required to compute the SG-value of a {\sc nim} position that allows one pass.

Unfortunately, closed formula for SG-values of positions in \textsc{nim} with a pass is \imp{unknown}. However, by using this result, we can expect faster analysis than with brute-force algorithms.
In particular, we have closed formulas for small SG-values in two-pile \textsc{nim} with a pass \imp{as follows}. These results can be used for analyzing games, like some \textsc{chocolate games} demonstrated in the next section. 


Let $\mathcal{G}_\mathrm{P}(x, y) = \mathcal{G}(C_{(+, \mathrm{pass})}^\mathrm{nim}(x,y))$. Table \ref{tab:nimpass} shows the values of $\mathcal{G}_\mathrm{P}(x, y)$.For the cases that $\mathcal{G}_\mathrm{P}(x, y)$ is small, we have the following closed formulas. For the proofs of them, see the appendix.

\begin{table}[tb]
    \caption{Values of $\mathcal{G}_\mathrm{P}(x, y)$.}
    \label{tab:nimpass}

    \centering
    \begin{tabular}{c|ccccccccccccc}
        $x \backslash y$ & 0 & 1 & 2 & 3 & 4 & 5 & 6 & 7 & 8 & 9 & 10 & 11 & 12 \\ \hline
        0 & 0 & 2 & 1 & 4 & 3 & 6 & 5 & 8 & 7 & 10 & 9 & 12 & 11 \\
        1 & 2 & 1 & 0 & 3 & 4 & 5 & 6 & 7 & 8 & 9 & 10 & 11 & 12 \\
        2 & 1 & 0 & 2 & 5 & 7 & 3 & 8 & 4 & 6 & 12 & 11 & 10 & 9\\
        3 & 4 & 3 & 5 & 1 & 0 & 2 & 7 & 6 & 9 & 8 & 12 & 13 & 10 \\
        4 & 3 & 4 & 7 & 0 & 1 & 8 & 9 & 2 & 5 & 6 & 13 & 14 & 15 \\
        5 & 6 & 5 & 3 & 2 & 8 & 1 & 0 & 9 & 4 & 7 & 14 & 15 & 13 \\
        6 & 5 & 6 & 8 & 7 & 9 & 0 & 1 &  3 & 2 & 4 & 15 & 16 & 14 \\
        7 &8 & 7 & 4 & 6 & 2 & 9 & 3 & 1 & 0 & 5 & 16 & 17 & 18 \\
        8 & 7 & 8 & 6 & 9 & 5 & 4 & 2 & 0 & 1 & 3 & 17 & 18 & 16 \\
        9 & 10 & 9 & 12 & 8 & 6 & 7 & 4 & 5 & 3 & 1 & 0 & 19 & 2 \\
        10 & 9 & 10 & 11 & 12 & 13 & 14 & 15 & 16 & 17 & 0 & 1 & 2 & 3 \\
        11 & 12 & 11 & 10 & 13 & 14 & 15 & 16 & 17 & 18 & 19 & 2 & 1 & 0 \\
        12 & 11 & 12 & 9 & 10 & 15 & 13 & 14 & 18 & 16 & 2 & 3 & 0 & 1 \\
                
    \end{tabular}
\end{table}

\begin{lemma}
\label{lem:gv0}
    Let $x$ and $y$ be nonnegative integers.
\begin{enumerate}
\item[(a)]    $\mathcal{G}_\mathrm{P}(x, y) = 0$ if and only if one of the following holds; 
    \begin{itemize}
        \item $x = y = 0$.
        \item $x = 2n - 1$ and $y = 2n$ for a positive integer $n$.
        \item $x = 2n$ and $y = 2n - 1$ for a positive integer $n$.
        
    \end{itemize}

\item[(b)]   $\mathcal{G}_\mathrm{P}(x, y) = 1$ if and only if one of the following holds; 
    \begin{itemize}
        \item $(x, y) \in \{(0, 2), (2, 0)\}$.
        \item $x = y, x \neq 0,$ and $ x \neq 2$.
    \end{itemize}

\item [(c)]    $\mathcal{G}_\mathrm{P}(x, y) = 2$ if and only if one of the following holds; 
    \begin{itemize}
        \item $(x, y) \in \{(0, 1), (1, 0), (2, 2), (3, 5), (4, 7), (5, 3), (6, 8), (7, 4), (8, 6)\}$.
        \item $(x - 1) \oplus (y - 1) = 3$ and $x, y \geq 9$.
    \end{itemize}
\end{enumerate}
\end{lemma}

\section{Application for three-dimensional \textsc{chocolate game}}
\label{sec:choco}

In this section, we show how Theorem \ref{thm:pass} can be used for analyzing games. 

Let us consider a \textsc{chocolate game} in Definition \ref{defofbarwithfunc3d} under the condition that $F(x,z)= h(z)$, where $h$ satisfies NS-property. 
We call this game \textsc{stair chocolate game}. Some examples of positions in the game are shown in Figures \ref{3Dchocolate22}, \ref{3D337}, \ref{3D527}, and \ref{3Dchocolate}. Then, \imp{from Theorem \ref{sufficientcond3d},} the SG-value of the chocolate bar $CB(F,x,y,z)$ in the game is $x \oplus y \oplus z$. 

Note that since $F(x, z) = h(z)$, the value of $x$ is independent from $y$ and $z$. Therefore, we can consider this ruleset as a disjunctive compound of one-heap \textsc{nim} and two-dimensional \textsc{chocolate game}. We \imp{can} use this construction for analyzing \textsc{stair chocolate game} with a pass move.  The position of this game is represented by four coordinates $(x,y,z,p)$, where $p = 1$ if the pass is still available; otherwise, $p = 0$.

We already have  small SG-values of two-dimensional chocolate games where $h$ has NS-property in Lemma \ref{lem:leq8}. We also have  small SG-values of two-pile \textsc{nim} with a pass in Lemma \ref{lem:gv0}. 
By combining these results using Theorem \ref{thm:pass}, we have the following corollaries.
Since we do not have any formula for SG-values or $\mathcal{P}$-positions of three-pile \textsc{nim} with a pass, it is remarkable that we have some formulas for SG-values including $\mathcal{P}$-positions of these three-dimensional \textsc{chocolate games}. For the proofs of these corollaries, see the appendix.

\begin{corollary}
\label{cor:g0}
Let 
\begin{eqnarray*}
A_0 &=& \{(x,y,z,0)\mid   x \oplus y \oplus z=0, y\leq h(z) \}, \\
A_1&=&\{(x,y,z,1)\mid  x \mbox{ is odd}, (x+1)  \oplus  y   \oplus   z =0, y\leq h(z) \}, \\  
A_2&=& \{(x,y,z,1)\mid  x \geq 2, x \mbox{ is even}, (x-1)  \oplus y   \oplus  z =0, y\leq h(z) \},  \text{and}\\
A&=&A_0 \cup A_1 \cup A_2 \cup \{(0,0,0,1) \}.
\end{eqnarray*}
A position $(x, y, z, p)$ is in  $A$ if and only if the position is a 
 $\mathcal{P}$-position.
\end{corollary}

\begin{corollary}
\label{cor:g1}
Let 
\begin{eqnarray*} 
B_0&=&\{(x,y,z,0)\mid x \oplus y \oplus z = 1, y \leq h(z)\}, \\
B_1 &=&\{(x,y,z,1)\mid(x, y, z)\in \{(0, 0, 2),  (0, 1, 3),  (2, 0, 0)\}, y \leq h(z)\}, \\
B_2 &=&\{(x,y,z,1)\mid x \oplus y \oplus z = 0,  y \leq h(z) \}, \\
B_3&=&\{(x,y,z,1)\mid(x, y, z)\in \{(0,0,0), (2,0,2),(2,1,3)\}, y\leq h(z)\}, \text{and} \\ 
 B&=&B_0 \cup B_1 \cup B_2-B_3.
 \end{eqnarray*}
A position $(x, y ,z, p)$ is in $B$ if and only if the position has SG-value $1$.
\end{corollary}

\begin{corollary}
\label{cor:g2}

Let 
\begin{eqnarray*}
 C_0 &=& \{(x, y, z, 0)\mid x \oplus y \oplus z = 2, y\le h(z)\}, \\
C_1 &=&\{(x,y,z,1)\mid (x,y,z)\in \{  (0, 0, 1), (1, 0, 0), (2, 0, 2), (2, 1, 3), (3, 0, 5) ,  
 (3, 1, 4),(3, 2, 7), \\ & &(3, 3, 6), (4, 0, 7), (4, 1, 6), (4, 2, 5), (4, 3 ,4), (5, 0, 3), \ (5, 1, 2), (6, 0, 8), (6, 1, 9), \\ & &(6, 2, 10), (6, 3, 11), (6, 4, 12),(6, 5, 13), (6, 6, 14), (6, 7, 15), (7, 0, 4), (7, 1, 5), (7, 2, 6), \\ & &(7, 3, 7), (8, 0, 6), 
  (8, 1, 7),(8, 2, 4), (8, 3, 5) \}, y \leq h(z) \},  \\
C_2 &=& \{(x, y, z, 1)\mid (((x - 1) \oplus 3) + 1) \oplus y \oplus z = 0, y \le h(z)\} \\
C_3 &=& \{(x,y,z,1)\mid (x,y,z)\in \{  (1, 0, 4), (1, 1, 5), (1, 2, 6), (1, 3, 7), (2, 0, 3),   (2, 1, 2), (3, 0, 2),  \\ & &(3, 1, 3), (4, 0, 1), (5, 0, 8), (5, 1, 9), (5, 2, 10), (5, 3, 11), (5, 4, 12),   (5, 5, 13),(5, 6, 14), \\ & &(5, 7, 15), (6, 0, 7), (6, 1, 6), (6, 2, 5), (6, 3, 4),(7, 0, 6), (7, 1, 7), (7, 2, 4), (7, 3, 5), (8, 0, 5), \\ & &(8, 1, 4), (8, 2, 7),  (8, 3, 6)\}, y\leq h(z) \}, \text{and} \\
 C &=& C_0 \cup C_1 \cup C_2 - C_3.
\end{eqnarray*}
A position $(x, y, z, p)$ is in $C$ if and only if the position has SG-value $2$.
\end{corollary}

\subsubsection{\ackname} 
The authors have no competing interests to declare that are
relevant to the content of this article. The authors extend their gratitude to Mr. Keito Tanemura, Mr. Yuji Sasaki, and Mr. Yuki Tokuni for their help on developing a calculation program for SG-values. 






%
%
%
 \bibliographystyle{splncs04}
 \bibliography{SOFSEM}

@article{robin,
    author = {Robin, A. C.},
    title = {A Poisoned Chocolate Problem},
    journal = {Problem Corner, The Mathematical Gazette},
    volume    = {73},
    number    = {466},
    pages     = {341--343},
    year      = {1989}
}

@article{gale,
    author = {Gale, D.},
    title = {A curious {N}im-type game},
    journal = {Math. Monthly},
    volume = {81},
    number = {8},
    pages = {876--879},
    year = {1974}
}

@book{lesson,
    author = {Albert, M. H. and Nowakowski, R. J. and Wolfe, D.},
    title = {Lessons In Play: An Introduction to Combinatorial Game Theory},
    edition = {Second},
    publisher = {A K Peters/CRC Press},
    address = {Boca Raton, Florida, United States},
    year = {2019}
}

@article{spr,
    author = {Sprague, R. P.},
    title = {\"{U}ber mathematische {K}ampfspiele},
    journal = {T\^{o}hoku Mathematical Journal},
    volume = {41},
    pages = {438–444},
    year = {1935–36}
}

@article{gru,
    author = {Grundy, P. M.},
    title = {Mathematics and games},
    journal = {Eureka},
    volume = {2},
    pages = {6--8},
    year = {1939}
}

@article{integer2021,
    author = {Miyadera, R. and Nakaya, Y.},
    title = {Grundy Numbers of Impartial Three-Dimensional Chocolate-Bar Games},
    journal = {Integers},
    volume = {21B},
    number = {\#A19},
    year = {2021}
}

@article{nimpass,
    author = {Morrison, R. E. and Friedman, E. J. and Landsberg, A. S.},
    title = {Combinatorial games with a pass: A
dynamical systems approach},
    journal = {Chaos, An Interdisciplinary Journal of Nonlinear Science},
    volume = {21},
    number = {4},
    year = {2011}
}

@article{LC15,
    author = {Low, R. M. and Chan, W. H.},
    title = {An atlas of {N}- and {P}-positions in `{N}im with a pass'},
    journal = {Integers},
    volume = {15},
    number = {\#G2},
    year =  {2015}
}

@article{integers1,
    author = {Miyadera, R. and Inoue, M. and Fukui, M.},
    title = {Impartial Chocolate Bar Games with a Pass},
    journal = {Integers},
    volume = {16},
    number = {\#G5},
    year = {2016}
}

@article{CLLW18,
    author = {Chan, W. H. and Low, R. M. and Locke, S. C. and Wong, O. L.},
    title = {A map of the {P}-positions in `{N}im {W}ith a {P}ass' played on heap sizes of at most four},
    journal = {Discrete Applied Mathematics},
    volume = {244},
    pages = {44--55},
    year = {2018}
}

@article{Hirsch20,
    author = {Hirsch, E.},
    title = {Investigations of Impartial Games With a Pass},
    journal = {preprint},
    volume = {arXiv:2010.10643},
    year = {2020}
}

@article{bou,
    author = {C. L. Bouton},
    title = {Nim, a game with a complete mathmatical theory},
    journal = {Ann. of Math.},
    volume = {3},
    pages = {35--39},
    year = {1902}
}

@article{hyper,
    author = {Boros, E. and Gurvich, V. and Ho, N. B. and Makino, K. and Mursic, P},
    title = {Impartial games with decreasing {S}prague–{G}rundy function and their hypergraph compound},
    journal = {Int J Game Theory},
    volume = {53},
    pages = {1119--1144},
    year = {2024}
}

@article{exact,
author = {Boros, E. and Gurvich, V. and Ho, N. B. and Makino, K. and Mursic, P.},
title = {On the {S}prague–{G}rundy function of {\sc Exact $k$-{N}im}},
journal = {Discrete Applied Mathematics},
volume = {239},
pages = {1--14},
year = {2018}}

@article{minnim,
    author = {L. Levine},
    title = {Fractal Sequences and Restricted Nim},
    journal = {Ars Combinatoria},
    volume = {LXXX},
    pages = {113--128},
    year = {2006}
}

@article{unsolve,
    author = {R. J. Nowakowski},
    title = {Unsolved problems in combinatorial games},
    journal = {Games of No Chance},
    volume = {V},
    pages = {125--168},
    year = {2019}
}

\appendix
\section{Omitted proofs}\label{sec:proofs}
\subsection{Proof of Proposition \ref{lemmaforfloorzbyk}}
\begin{enumerate}
\item
\imp{Assume that $\lfloor \frac{z}{2^i}\rfloor = \lfloor \frac{z^{\prime}}{2^i}\rfloor$. Then, }
there exists
$u \in \mathbb{Z}_{\geq 0}$ such that 
\begin{displaymath}
z=u\times 2^i + \sum\limits_{j = 0}^{i-1} {{z_j}} {2^j} \nonumber 
\end{displaymath}
and
\begin{displaymath}
z^{\prime}=u\times 2^i + \sum\limits_{j = 0}^{i-1} 
{{z^{\prime}_j}} {2^j}, \nonumber 
\end{displaymath}
where $z_j, z^{\prime}_j \in \{0,1\}$ for $j=0,1, \cdots, i-1.$
Then, there exist $s,t \in \mathbb{Z}_{\geq 0}$ such that $u=kt+s$ and $0 \leq s <k$, and we have 
\begin{displaymath}
0 \leq \frac{z}{2k}-t2^{i-1} =\frac{s}{k}2^{i-1}+ \frac{1}{k}(z_{i-1}2^{i-2}+z_{i-2}2^{i-3}+ \cdots + z_02^{-1}) <\frac{s+1}{k}2^{i-1} \leq 2^{i-1}. \nonumber 
\end{displaymath}
and
\begin{displaymath}
0 \leq \frac{z^\prime}{2k}-t2^{i-1} =\frac{s}{k}2^{i-1}+ \frac{1}{k}(z^{\prime}_{i-1}2^{i-2}+z^{\prime}_{i-2}2^{i-3}+ \cdots + z^{\prime}_02^{-1}) <\frac{s+1}{k}2^{i-1} \leq 2^{i-1}. \nonumber 
\end{displaymath}
Then, we have 
\begin{displaymath}
t \leq \left\lfloor \dfrac{\left\lfloor \dfrac{z}{2k}\right\rfloor}{2^{i-1}}\right\rfloor
, \left\lfloor \dfrac{\left\lfloor \dfrac{z^{\prime}}{2k}\right\rfloor}{2^{i-1}}\right\rfloor
< t+1, \nonumber 
\end{displaymath}
and hence 
\begin{displaymath}
\left\lfloor \frac{h(z)}{2^{i-1}}\right\rfloor =\left\lfloor \frac{\left\lfloor \frac{z}{2k}\right\rfloor}{2^{i-1}}\right\rfloor
= t =\left\lfloor \frac{\left\lfloor \frac{z^{\prime}}{2k}\right\rfloor}{2^{i-1}}\right\rfloor
=\left\lfloor \frac{h(z^{\prime})}{2^{i-1}}\right\rfloor, \nonumber 
\end{displaymath}
thus, we have 
\imp{$\lfloor \frac{h(z)}{2^{i-1}} \rfloor = \lfloor \frac{h(z^{\prime})}{2^{i-1}}\rfloor$.}
\item Suppose that 
$\lfloor \frac{z}{2^i} \rfloor = \lfloor \frac{z^{\prime}}{2^i} \rfloor$. Then,
$z=u \times 2^i + v$ and $z^{\prime}=u \times 2^i + v^{\prime}$
for some $u,v,v^{\prime} \in \mathbb{Z}_{\ge 0}$ such that $0 \leq v,v^{\prime} < 2^i$.
\begin{itemize}
\item[(i)] If $u >0$, there exists $t \in \mathbb{Z}_{\ge 0}$ such that 
$2^t \leq u < 2^{t+1}$.
Then, 
$2^t \leq u+\frac{v}{2^i}, u+\frac{v^{\prime}}{2^i}  < 2^{t+1}$, and hence
$2^{i+t} \leq 2^i(u+\frac{v}{2^i}), 2^i(u+\frac{v^{\prime}}{2^i}) < 2^{i+t+1}$.
Therefore, 
$\lfloor \log_2z \rfloor = i+t = \lfloor \log_2z^{\prime} \rfloor$, and 
we have $h(z)=h(z^{\prime})$.\\
\item[(ii)] If $u =0$, then $z, z^{\prime} < 2^i$.
Then, $\lfloor \log_2z \rfloor, \lfloor \log_2z^{\prime} \rfloor \leq i-1$, and hence
$ \lfloor \frac{2^{\lfloor \log_2z \rfloor}-1}{2^{i-1}} \rfloor$ 
$= \lfloor \frac{2^{\lfloor \log_2z^{\prime} \rfloor}-1}{2^{i-1}} \rfloor =0.$
Therefore, we have 
\begin{equation}
\left  \lfloor \frac{h(z)}{2^{i-1}}\right  \rfloor = \left  \lfloor \frac{h(z^{\prime})}{2^{i-1}}\right  \rfloor=0. \nonumber
\end{equation}
\end{itemize}
From (i) and (ii), in each case, $\left  \lfloor \dfrac{h(z)}{2^{i-1}}\right  \rfloor = \left  \lfloor \dfrac{h(z^{\prime})}{2^{i-1}}\right  \rfloor$ holds.
\qed
\end{enumerate}
\subsection{Proof of Lemma \ref{lem:gv0}}
\begin{enumerate}
\item[(a)]    We prove this by induction on $x+y$.
    From Definition \ref{def:dispas}, $\mathcal{G}_\mathrm{P} (0, 0) = 0$.
    Assume that for any positive integer $n' < n,$ the statement holds. Then for any $k < 2n - 1, \mathcal{G}_\mathrm{P}(2n, k) \neq 0, \mathcal{G}_\mathrm{P}(2n - 1, k) \neq 0, \mathcal{G}_\mathrm{P}(k, 2n) \neq 0,$ and $\mathcal{G}_\mathrm{P}(k, 2n-1) \neq 0$ since each $C^\mathrm{nim}_{(+, \mathrm{pass})}(2n, k), C^\mathrm{nim}_{(+, \mathrm{pass})}(2n-1, k), C^\mathrm{nim}_{(+, \mathrm{pass})}(k, 2n), $ and $C^\mathrm{nim}_{(+, \mathrm{pass})}(k, 2n-1)$ has an option whose SG-value is $0$.
    We also have $\mathcal{G}_\mathrm{P}(2n-1, 2n-1) \neq 0$ since  $C^\mathrm{nim}_{(+, \mathrm{pass})}(2n-1, 2n-1)$ has an option to $C^\mathrm{nim}_{+}(2n-1, 2n-1),$ whose SG-value is $(2n-1) \oplus (2n - 1) = 0.$ 
    
    Therefore, we have $\mathcal{G}_\mathrm{P}(2n, 2n -1) = 0.$
    Similarly, $\mathcal{G}_\mathrm{P}(2n - 1, 2n) = 0$ holds.
    
\item [(b)] We prove that $\mathcal{G}_\mathrm{P}(x, y) = 1$ if and only if $x + y  =2 $ or $x = y > 2$ holds. This statement is equivalent to this lemma. We use induction on $x+y$.
    
    It is easy to confirm that $\mathcal{G}_\mathrm{P}(x, y) = 1$ when $x + y = 2$. This means that $\mathcal{G}_\mathrm{P}(x, y) \neq 1$ when $(x, y) \not\in \{(0, 2), (1, 1), (2, 0)\}$ and $\min(\{x, y\}) \leq 2$. 
    Consider the case $\min(\{x, y\}) > 2$. Assume that for any $n' < n$, $\mathcal{G}_\mathrm{P}(n', n') = 1$. Then, for any $k \neq n', \mathcal{G}_\mathrm{P}(n', k) \neq 1$ and $\mathcal{G}_\mathrm{P}(k, n') \neq 1$ hold. 
    
    Since $n \oplus n = 0$,  we have  $\mathcal{G}_\mathrm{P}(n, n) = \mathrm{mex}(\{\mathcal{G}_\mathrm{P}(n', n) \mid n' < n\} \cup \{\mathcal{G}_\mathrm{P}(n, n')\mid n' < n\} \cup \{0\}) = 1$. 

\item [(c)]   We prove this by induction on $x+y$.
    Note that when $(x - 1) \oplus (y - 1) = 3,$ we have $(x, y) = (4n+1, 4n+4), (4n+2, 4n + 3), (4n+3, 4n+2),$ or $(4n + 4, 4n + 1)$ for an integer $n$.
    
    From Table \ref{tab:nimpass}, it is easy to confirm that when $(x,y) \in  \{(0, 1), (1, 0), (2, 2), \allowbreak(3, 5), (4, 7), (5, 3), \allowbreak(6, 8), (7, 4), (8, 6)\}$, $\mathcal{G}_\mathrm{P}(x, y) = 2$. This means that if $\min(\{x, y\}) \leq 8$ and $(x, y) \not \in \{(0, 1), (1, 0),\allowbreak (2, 2), (3, 5), (4, 7), (5, 3), (6, 8), (7, 4), (8, 6)\}, \mathcal{G}_\mathrm{P}(x, y) \neq 2.$

    Let $n \geq 2$.
    Assume that for any $n' < n, $ $ \mathcal{G}_\mathrm{P}(4n'+1, 4n'+4) = \mathcal{G}_\mathrm{P} (4n'+2, 4n' + 3) = \mathcal{G}_\mathrm{P} (4n'+3, 4n'+2) = \mathcal{G}_\mathrm{P}(4n' + 4, 4n' + 1) = 2.$ Then, we have $\mathcal{G}_\mathrm{P}(4n'+1, k) \neq 2$ for any $k \neq 4n' + 4, \mathcal{G}_\mathrm{P}(4n'+2, k) \neq 2$ for any $k \neq 4n' + 3, $
    $\mathcal{G}_\mathrm{P}(4n'+3, k) \neq 2$ for any $k \neq 4n' + 2, \mathcal{G}_\mathrm{P}(4n'+4, k) \neq 2$ for any $k \neq 4n' + 1, \mathcal{G}_\mathrm{P}(k, 4n'+1) \neq 2$ for any $k \neq 4n' + 4, \mathcal{G}_\mathrm{P}(k, 4n'+2) \neq 2$ for any $k \neq 4n' + 3, \mathcal{G}_\mathrm{P}(k, 4n'+3) \neq 2$ for any $k \neq 4n' + 2, $ and 
    $\mathcal{G}_\mathrm{P}(k, 4n'+4) \neq 2$ for any $k \neq 4n' + 1.  $
    
    We also have $\mathcal{G}_\mathrm{P}(4n + 1, 4n + 1) = \mathcal{G}_\mathrm{P}(4n + 2, 4n + 2) = 1$ from Lemma \ref{lem:gv0} (b), $\mathcal{G}_\mathrm{P}(4n + 1, 4n + 2) = \mathcal{G}_\mathrm{P}(4n + 2, 4n + 1) = 0$ from Lemma \ref{lem:gv0} (a), $\mathcal{G}_\mathrm{P}(4n + 1, 4n + 3) \neq 2$ since $(4n + 1) \oplus (4n + 3) = 2$, $\mathcal{G}_\mathrm{P}(4n + 3, 4n + 1) \neq 2 $ since $(4n + 3) \oplus (4n + 1) = 2$, $(4n + 1) \oplus (4n + 4) = (4n + 4) \oplus (4n + 1) = 5,$ and $(4n + 2) \oplus (4n + 3) = (4n + 3) \oplus (4n + 2) = 1$. From them, we have $\mathcal{G}_\mathrm{P}(4n + 1, 4n + 4) = \mathcal{G}_\mathrm{P}(4n + 2, 4n + 3) = \mathcal{G}_\mathrm{P}(4n+3, 4n+2) = \mathcal{G}_\mathrm{P}(4n+ 4, 4n + 1) = 2.$ \qed 
\end{enumerate}

\subsection{Proof of Corollary \ref{cor:g0}}
    A position in \textsc{stair chocolate game} can be considered as a disjunctive (resp. one-pass) compound of two positions in different rulesets when there is no (resp. one) pass-move; one component $G(x)$ is one-pile \textsc{nim} with $x$ tokens and the other component $H(y,z)$ is two-dimensional \textsc{chocolate game} $CB_2(h, y, z)$. 
    In the two-dimensional \textsc{chocolate game}, from every position except for the terminal position, one can move to the terminal position. Also, in one-pile \textsc{nim}, 
    one can move to the terminal position from any non-terminal position. Therefore, from Theorem \ref{thm:pass} and Lemma \ref{lem:gv0} (a), the position $C_{(+, \mathrm{pass})}(G(x), H(y,z)) $ is a $\mathcal{P}$-position if and only if one of the followings holds: 
    \begin{itemize}
        \item[(i)] $\mathcal{G}(G(x)) = \mathcal{G}(H(y, z)) = 0$.
        \item[(ii)] $\mathcal{G}(G(x)) = \mathcal{G}(H(y,z))- 1$ and $\mathcal{G}(G(x))$ is a positive odd integer.
        \item[(iii)] $\mathcal{G}(G(x)) = \mathcal{G}(H(y,z)) + 1$ and $\mathcal{G}(G(x))$ is a positive even integer.
    \end{itemize}
    Condition (i) is a necessary and sufficient condition for $(x,y ,z, p)= (0, 0, 0, 1)$.

    Condition (ii) is a necessary and sufficient condition for $(x, y, z, p) \in A_1.$ 

    Condition (iii) is a necessary and sufficient condition for $(x, y, z, p) \in A_2.$ 

    In addition, $C_+(G(x), H(y,z))$ is a $\mathcal{P}$-position if and only if  $\mathcal{G}(G(x)) = \mathcal{G}(H(y,z))$.      
    This is a necessary and sufficient condition for $(x, y, z, p) \in A_0$ 

    Therefore, a position $(x, y, z, p)$ is a $\mathcal{P}$-position if and only if $(x, y, z, p) \in A$. \qed

\subsection{Proof of Corollary \ref{cor:g1}}
    Similar to the proof of Corollary \ref{cor:g0}, we can use Theorem \ref{thm:pass} for the position $C_{(+, \mathrm{pass})}(G(x), H(y,z)) $ where $G(x)$ is a position in one-pile \textsc{nim} and $H(y,z)$ is a position in two-dimensional \textsc{chocolate game}. 
    From Lemma \ref{lem:gv0} (b), the position has SG-value $1$ if and only if one of the followings holds: 
    \begin{itemize}
        \item[(i)] $(\mathcal{G}(G(x)),   \mathcal{G}(H(y,z)))  \in \{(0, 2) , (2, 0)\}.$
        \item[(ii)] $\mathcal{G}(G(x)) = \mathcal{G}(H(y,z))$, $\mathcal{G}(G(x)) \neq 0, $ and $ \mathcal{G}(G(x)) \neq 2$.
    \end{itemize}
    Condition (i) is a necessary and sufficient condition for $(x, y ,z, p) \in B_1$ because from Lemma \ref{lem:leq8}, $H(y,z)$ has SG-value $2$ if and only if $(y, z) \in \{ (0, 2), (1, 3)\}$.

    Condition (ii) is a necessary and sufficient condition for $(x, y ,z, p) \in B_2 - B_3$.

    In addition, $C_+(G(x), H(y, z))$ has SG-value $1$ if and only if $\mathcal{G}(G(x)) \oplus \mathcal{G}(H(y,z)) = x\oplus (y\oplus z) = 1$. That is, $(x, y, z, p) \in B_0$.

    Therefore, $(x, y, z, p) \in B$ if and only if the position has SG-value $1$. \qed
\subsection{Proof of Corollary \ref{cor:g2}}


Similar to the proofs of Corollaries \ref{cor:g0} and \ref{cor:g1}, we can use Theorem \ref{thm:pass} for the position $C_{(+, \mathrm{pass})}(G(x), H(y,z))$ where $G(x)$ is a position in one-pile \textsc{nim} and $H(y,z)$ is a position in two-dimensional \textsc{chocolate game}. From Lemma \ref{lem:gv0} (c), the position has SG-value $2$ if and only if one of the following holds: 

\begin{itemize}
    \item[(i)] $(\mathcal{G}(G(x)), \mathcal{G}(H(y,z))) \in \{(0, 1), (1, 0), (2, 2), (3, 5), (4, 7), (5, 3), (6, 8), (7, 4), (8, 6)\}$.
        \item[(ii)] $(\mathcal{G}(G(x)) - 1) \oplus (\mathcal{G}(H(y,z)) - 1) = 3$ and $\mathcal{G}(G(x)), \mathcal{G}(H(y,z)) \geq 9$.
\end{itemize}
From Lemma \ref{lem:leq8}, Condition (i) is a necessary and sufficient condition for $(x, y, z, p) \in C_1$. 

Consider Condition (ii). If $(a - 1) \oplus (b - 1) = 3,$ then $b =  ((a- 1) \oplus 3)+1$  holds. Therefore, Condition (ii) is a necessary and sufficient condition for $(x, y, z, p) \in C_2- C_3$.

In addition, $C_+(G(x), H(y,z))$ has SG-value $2$ if and only if $\mathcal{G}(G(x)) \oplus \mathcal{G}(H(y,z)) = x \oplus (y \oplus z) = 2$. That is, $(x, y, z, p) \in C_0$.

Therefore, $(x, y, z, p) \in C$ if and only if the position has SG-value $2$. \qed

\end{document}